\documentclass[12pt,english]{amsart}
\usepackage{pxfonts}
\usepackage{amsmath}
\usepackage{amsthm}
\usepackage{amsfonts}
\usepackage{amssymb}
\usepackage{amscd}
\usepackage[dvips]{graphicx}
\usepackage{epstopdf}
\usepackage{psfrag}

\theoremstyle{plain}
\newtheorem{mainthm}{Theorem}

\swapnumbers
\newtheorem{theorem}{Theorem}[section]
\newtheorem{proposition}[theorem]{Proposition}
\newtheorem{corollary}[theorem]{Corollary}
\newtheorem{lemma}[theorem]{Lemma}

\newtheorem{question}[theorem]{Question}

\newtheorem{remark}[theorem]{Remark}

\newcommand{\Z}{\mathbb{Z}}
\newcommand{\N}{\mathbb{N}}
\newcommand{\R}{\mathbb{R}}
\newcommand{\eps}{\varepsilon}

\newcommand{\dm}{\diff^{1}_{m}(M)}
\newcommand{\dw}{\diff^{1}_{\omega}(M)}

\newcommand{\diff}{\operatorname{Diff}}

\newcommand{\dis}{\displaystyle}
\newcommand{\vs}{\vspace{0.1cm}}
\newcommand{\vsm}{\vspace{0.5cm}}

\title[A lower bound for topological entropy]{A lower bound for topological entropy of generic non Anosov symplectic diffeomorphisms}
\author{Thiago Catalan and Ali Tahzibi}
\thanks{
Partially supported by CAPES and FAPESP.}

\date{\today}

\begin{document}

\maketitle

\begin{abstract}
We prove that a $C^1-$generic symplectic diffeomorphism is either Anosov or the topological entropy is bounded from below by the supremum over
the smallest positive Lyapunov exponent of the periodic points.  We also prove that $C^1-$generic symplectic diffeomorphisms outside
the Anosov ones do not admit symbolic extension and finally we give examples of  volume preserving diffeomorphisms which
are not point of upper semicontinuity of entropy function in $C^1-$topology.
\end{abstract}

\section{Introduction}

The topological entropy  is one of the most important topological invariants for dynamical systems. Informally, the topological entropy
calculates the ``number of different trajectories" of the dynamics. Formally, we  define it in the following way
$$
h(f)= \lim_{\eps\rightarrow 0} \limsup_{n\rightarrow \infty} \frac{1}{n}\log r(n,\eps);
$$
where $r(n,\eps)$ is the maximum amount of $\eps$-distinct orbits of length
$n$. Two points have $\eps$-distinct orbits of length $n$ if there is  $0\leq j\leq n$ such that $d(f^j(x), f^j(y)) > \eps.$

For Axiom A diffeomorphisms,  Bowen \cite{Bowen}  proved that entropy determines the asymptotical exponential growth of the number of periodic
points and by a Katok's result for any $C^{1+\alpha}$ ($\alpha > 0$) diffeomorphism of a two dimensional ma\-ni\-fold entropy is bounded above
by such growth rate: $h(f) \leq \limsup_{n \rightarrow \infty} \frac{P_n(f)}{n}.$

In this paper we  prove lower estimates for topological entropy of $C^1-$generic symplectic dynamics in terms of Lyapunov exponents of periodic
points of the system ( See theorems \ref{symplectic}, \ref{equality}).

We relate such lower bounds for the entropy to the (semi continuity) regula\-ri\-ty of the entropy function with respect to the dynamics (see
Theorem \ref{example}) and construct examples of  surface diffeomorphisms which are not  point of semi continuity of the topological entropy.

  Finally, we take benefit of the estimates in order to prove the non existence of symbolic extensions for $C^1-$generic symplectic diffeomorphisms far from the Anosov ones ( See theorem \ref{symbolic}).

\subsection{Lower estimates for topological entropy}
Firstly, Newhouse in \cite{Newhouse1} got some lower bound for topological entropy of generic area preserving diffeomorphisms on surfaces far
from Anosov diffeomorphisms. Remember for a $C^1$ diffeomorphism over some manifold $M$ be Anosov means that the whole manifold $M$ is a
hyperbolic set for $f$, where an $f-$invariant compact set $\Lambda$ of $M$ is a {\it hyperbolic set} if there is a continuous and
$Df$-invariant splitting $T_{\Lambda}M=E^s\oplus E^u$ such that there are constants $0<\lambda<1$ and $C>0$, satisfying
$$
\|Df_x^k| E^s(x)\|\leq C\lambda^k \quad \text{and}\quad \|Df_x^{-k}| E^u(x)\|\leq C\lambda^k,
$$
for every $x\in \Lambda$ and $k>0$.

More precisely, let $M$ be a compact, connected surface with a volume form $m$, and  denote by $\dm$ the set of conservative $C^1$
diffeomorphisms, i.e., the set formed by diffeomorphisms that preserve the volume form $m$.  Taking
$$s(f)=\sup\left\{\dis\frac{1}{\tau(p,f)}\log \lambda(p,f)\right\}$$ over all hyperbolic periodic points $p$ of $f$, where  $\tau(p,f)$ is the
minimum period of the hyperbolic  periodic point $p$, and ${\lambda}(p,f)$ is the absolute value of the unique eigenvalue of $Df^{\tau(p,f)}(p)$
with absolute value larger than one, Newhouse's result is the following.

\begin{theorem} (Newhouse)
There exists a residual subset $\mathcal{B}\subset\diff_m^1(M)$ such that if $f\in\mathcal{B}$ is a non Anosov diffeomorphism then
$$
h(f)\geq s(f).
$$
\label{new1}
\end{theorem}

Here we show that indeed generically the reverse inequality also holds and this implies:
\begin{mainthm} \label{equality}
There exists a residual subset $\mathcal{B}\subset\dm$ (volume preserving surface diffeomorphisms) such that if $f\in\mathcal{B}$ is a non
Anosov diffeomorphism then
$$
h(f) = s(f).$$
\end{mainthm}

Observe as a corollary of this theorem and semi continuity of $f \rightarrow s(f)$ (see preliminary definitions in the next section) we conclude
that ``generically" topological entropy is semi continuous in $C^1-$topology. However, it is not known whether the semi-continuity points of
topological entropy form a $C^1-$generic subset.

It is interesting to mention that among Anosov diffeomorphisms, diffeomorphisms of a $C^1-$open and dense subset satisfy $h(f) < s(f).$ See
Proposition \ref{anosov} which gives an upper bound for the entropy of  Anosov diffeomorphisms.

In general it may exist diffeomorphisms where $h(f) > s(f).$  In fact there exist  even minimal diffeomorphism with positive
entropy. In two dimensional case there are minimal homeomorphisms with positive entropy \cite{Rees}. However, these examples are not volume preserving (and should be outside a $C^1-$ generic subset).

\begin{question}
Is there an example of a conservative $C^1$ surface diffeomorphism where $h(f) > s(f)?$
\end{question}

Our next result is a generalization of Newhouse's theorem to higher dimensional symplectic setting. Let $(M,\omega)$ be a compact, connected,
smooth Riemanian symplectic manifold. For a hyperbolic periodic point $p$ of $f$, we denote by  $\lambda(p,f)$ the absolute value of the
smallest eigenvalue of $Df^{\tau(p,f)}(p)$  between those ones with absolute value bigger than one. Define
$$s(f):=\sup\left\{\dis\frac{1}{\tau(p,f)}\log\lambda(p,f)\right\}$$ over all hyperbolic periodic points $p$ of $f$, and then what we prove is the
following.

\begin{mainthm} \label{symplectic}
There exists a residual subset $\mathcal{B}\subset\dw$ such that if $f\in\mathcal{B}$ is a non Anosov diffeomorphism then
$$
h(f)\geq s(f).
$$
\end{mainthm}

\begin{question}
What about Conservative case? What can be said if we define $s(f)$ as the supremum over the sum of all positive Lyapunov exponents of periodic
points instead of just the smallest Lyapunov exponent of periodic points?

\end{question}

\subsection{Regularity of Entropy}

An important problem in smooth ergodic theory is the regularity of entropy with respect to dynamics. By Newhouse result we know that $f
\rightarrow h(f)$ is  upper semicontinuous in the $C^{\infty}$ topology for any compact boundaryless manifold and using Katok's result it is
indeed continuous for  $C^{\infty}$ surface diffeomorphisms. Here we show that

\begin{mainthm} \label{example}
There are examples of surface diffeomorphisms $ f_0 \in \diff^{\infty} (M)$ such that $f \rightarrow h(f)$ is not even upper semi continuous in
the $C^1-$topology at $f_0$.
\end{mainthm}

\subsection{Symbolic Extensions}
Symbolic dynamics play a crucial role in ergodic theory. It is a challenging problem to know whether a dynamics can be codified. We can obtain
upper bounds for the entropy by means of symbolic dynamics: A dynamical systems $(M,f)$ have a {\it symbolic extension} if there exist a
subshift $(Y,\sigma)$ and a surjective map $\pi:Y\rightarrow M$ such that $\pi\circ\sigma=f\circ\pi$. $(Y,\sigma)$ is called an {\it extension}
of $(M,f)$ and $(M,f)$ a {\it factor} of $(Y,\sigma)$. And so, if the system has some symbolic extension we gain directly an upper bound of the
topological entropy. Nevertheless, such estimate could be extreme. This way, one extension is called  {\it principal extension} if the map $\pi$
is such that $h_{\nu}(\sigma)=h_{\pi_{*}\nu}(f)$ for every $\sigma-$invariant measure $\nu\in\mathcal{M}(\sigma)$ over $Y$, where
$h_{\nu}(\sigma)$ is the metric entropy of $\sigma$ with respect to $\nu$.

Boyle, D. Fiebig,  U. Fiebig \cite{Boyle}  proved that asymptotically $h-$expansive diffeomorphisms have a principal symbolic extension. By a
result of Buzzi \cite{Buzzi} every $C^{\infty}$ diffeomorphism of compact manifold is asymptotically entropy expansive and consequently have a
principal symbolic extension.  Also, recently  D. Burguet \cite{Burguet} showed that every $C^2$ surface diffeomorphism have symbolic
extensions. These results give a positive partial answer to the conjecture of Downarowicz and Newhouse which expects symbolic extension for any
$C^r$ ($r \geq 2$) diffeomorphism.

Let us mention that Diaz, Fisher, Pac\'{i}fico and Vieitez \cite{diazfisher} proved that every $C^1$ partially hyperbolic diffeomorphism with a
nonhyperbolic central bundle that splits in a dominated way into 1-dimensional sub-bundles is asymptotically $h-$expansive and therefore has a
principal symbolic extension. See also \cite{df}.

On the other hand, Downarowicz and Newhouse using Theorem \ref{new1} proved in \cite{Newhouse3} that far from Anosov diffeomorphisms, generic
area preserving diffeomorphisms in $C^1$ topology admit no symbolic extensions. We extend this result to higher dimensional symplectic
diffeomorphisms.

\begin{mainthm} \label{symbolic}
There is a residual subset $\mathcal{B}\subset \dw$ such that if $f\in\mathcal{B}$ is a non Anosov diffeomorphism then $f$ has no symbolic
extension.
\end{mainthm}

Now, using this result we be able to give an easy and short prove of the stability conjecture in symplectic scneario. We say that a symplectic
diffeomorphism is structurally stable if there is some neighborhood $\mathcal{U}$ of $f$ in $\diff^1_{\omega}(M)$ such that every diffeomorphism
$g\in \mathcal{U}$ is topologically conjugated to $f$, i.e., there is one homeomorphism $h$ over $M$ such that $hf=gh$.

\begin{corollary}
A diffeomorphism $f\in \diff^1_{\omega}(M)$ is structurally stable if, and only if, $f$ is Anosov.
\end{corollary}

\proof Suppose $f\in \diff^1_{\omega}(M)$ is  a structurally stable diffeomorphism. Now, by Zehnder \cite{zehnder}, smooth diffeomorphisms are
dense among the symplectic ones and since $C^{\infty}$ diffeomorphisms have a principal symbolic extension, every diffeomorphism in some
neighborhood of $f$ also has a principal symbolic extension. So, accordingly to Theorem \ref{symbolic} this is only possible if $f$ is Anosov.

$\hfill\square$

\begin{remark}
The proof of the above corollary is based on Theorem \ref{symbolic}, but to prove it we use the unfolding of homoclinic tangency outside Anosov
set that happens in the symplectic scenario,  which is by itself an obstruction of stability. Nevertheless, we would like to emphasize that a
priori non existence of symbolic extensions does not have direct relation to homoclinic tangencies. So, we conclude that ``any mechanism in
$C^1-$topology" which yields non existence of symbolic extensions implies non structural stability.

However, very recently we were informed  that G.Liao, J.Yang and  M. Viana \cite{GYV} proved that diffeomorphisms $C^1$-far from tangencies have symbolic extensions.

\end{remark}

This paper is organized as follows.  In section 2  we will prove some generic results relating topological entropy with Lyapunov exponents of
periodic points, in particular, we prove Theorem A.  In section 3,  we prove Theorems B and C using a main technical proposition ( Proposition
\ref{proposição princial}).  In section 4, the main technical proposition is proved and finally in section 5 we give examples of (upper semi)
discontinuity points for topological entropy.

\section{Entropy and Lyapunov exponents of periodic points}

In this section firstly we review some background definitions and results, and prove an strict upper bound for the entropy in a $C^1-$open and
dense subset of Anosov volume preserving diffeomorphisms. After, using results of Abdenur, Bonatti and Crovisier \cite{abc} we prove an upper
bound for entropy of $C^1-$generic volume preserving diffeomorphisms of compact manifolds (any dimension) and apply it to prove Theorem
\ref{equality}.

\subsection{Preliminary definitions}

Given $f\in\dw$ and a hyperbolic periodic point $p$ of $f$, we denote by $\chi(p,f)$ the smallest positive Lyapunov exponent for the hyperbolic
periodic point $p$ of $f$, i.e., $ \chi(p,f)=1/\tau(p,f) \,\log \lambda(p,f), $ where $\lambda(p,f)= (\sigma(Df^{-\tau(p,f)}|E^u))^{-1}$, being
$\sigma$ the spectral radio of the map, and as before $\tau(p,f)$ is the minimum period of the hyperbolic periodic point $p$. In fact, defining
$$\mu_p=\frac{1}{\tau(p,f)}\sum_{i=0}^{\tau(p)-1}\delta_{f^i(p)},$$ as the {\it periodic measure} for $p$, where $\delta_{f^i(p)}$ is the dirac
measure for $f^i(p)$, we have that $\chi(p,f)$ is the smallest positive lyapunov exponent for the ergodic measure $\mu_p$.

Now, given $n\in \N$ we consider $ s_n(f)=\max\left\{\chi(p,f); \quad p\in H_n(f)\right\}, $ where $H_n(f)$ is the set of hyperbolic periodic
points of period smaller or equal than $n$. Since $H_n(f)\subset H_{n+1}(f)$, we have $s_n(f)\leq s_{n+1}(f)$, and then it's well defined $
s(f)=\lim_{n\rightarrow\infty}s_n (f). $ From robustness of the hyperbolic periodic points we have that the functional $s_n$ is continuous for
every $n\in\N$, which implies that $s(f)$ is lower semicontinuous.

\subsection{Generic upper bound for entropy of Anosov difeomorphisms}

In the setting of volume preserving Anosov diffeomorphisms the scenario is much more clear.
\begin{proposition} \label{anosov}
There exists a $C^1-$open and dense subset $\mathcal{F}$ of volume preserving Anosov diffeomorphisms such that for any $f \in \mathcal{F}$ with
$\dim(E^u) =u$ we have
$$
h(f) < \sup_{p \in Per(f)} \sum_{i=1}^{u} \chi^+_i(p)
$$
where the sum is over all positive Lyapunov exponents of $p.$
\end{proposition}

\begin{proof}

Take any volume preserving Anosov diffeomorphism $f$. After a small $C^1-$perturbation, if necessary, we can assume that $f$ is a $C^2$ Anosov
diffeomorphism. This regularization result is due to Avila \cite{avila2008}. Now, we know that for volume preserving Anosov diffeomorphisms the
Lebesgue measure $m$ is the unique ergodic e\-qui\-li\-brium state for the potencial $\phi^u(.)=-\log J^u(f)$ where $J^u(f) := |det Df|E^u(f)|.$
Recall that the entropy maximizing measure is just the e\-qui\-li\-brium state for the identically zero potential. Using Bowen's result it is
clear that $\mu$, the entropy maximizing measure, coincides with Lebesgue if, and only if, the potential $\phi^u$ is cohomologous to a constant
function. So, perturbing $f$ in $C^1-$topology we can assume that $\mu$ is singular with respect to the Lebesgue measure.

Following, recall that in Bowen's approach the entropy maximizing measures are obtained as the limit of periodic distributions. That is $$ \mu_n
:= \frac{\sum_{p \in Per_n (f)} \delta_p}{\# Per_n(f)} \rightarrow \mu$$

Since $\phi^u(.)$ is a continuous function we have $\int -\phi^u(x) d\mu(x) = \lim_{n \rightarrow \infty} \int -\phi^u (x) d \mu_n(x)$ and by
definition we conclude that $ \int -\phi^u d\mu_n \leq \sup_{p \in Per_n(f)} \sum_{i=1}^{u} \chi^+_i(p). $
 Then, $$ \int -\phi^u d\mu \leq \sup_{p \in Per(f)} \sum_{i=1}^{u} \chi^+_i(p).$$

Now, provided $f$ is Anosov the pressure of $\phi^u (.)$ is zero. Hence,
\begin{align*}
0=P_f(\phi^u)&=h_m(f)+ \int \phi^u \;dm
\\
&> h_\mu(f)+\int \phi^u \; d\mu
\\
&=h(f) + \int \phi^u \; d\mu.
\end{align*}
So it came out that
$$
h(f) <\int -\phi^u\; d\mu \leq \sup_{p \in Per(f)} \sum_{i=1}^{u} \chi^+_i(p).
$$

Finally, we claim that any $C^1-$perturbation of $f$ also satisfies a similar ine\-qua\-li\-ty. Indeed, as entropy is locally constant (by
structural stability of Anosov diffeomorphisms) and $f \rightarrow \sup_{p \in Per(f)} \sum_{i=1}^{u} \chi^+_i(p)$ is a lower semi continuous
function we conclude that for any $g$ $C^1-$close enough to $f$ we have
$$
 h(g) < \sup_{p \in Per(g)} \sum_{i=1}^{u} \chi^+_i(p).
$$

\end{proof}

\subsection{Generic upper bounds  for entropy and proof of Theorem A}

Let $f$ be a $C^{1+\alpha}$ diffeomorphism on a compact manifold and $\mu$ an ergodic hyperbolic measure. Then Katok proved, see \cite{katok},
that there exists a sequence of periodic points $p_n$ such that dirac measures on the orbit of $p_n$ converge to $\mu$ and the Lyapunov
exponents of $p_n$ converge to the Lyapunov exponent of $\mu.$ By variational principle $h (f) = \sup_{\mu} h_{\mu}(f)$ where the supremum is
over all f-invariant ergodic probability measures. By Ruelle's inequality,   $h_{\mu}(f) \leq \sum \chi_i^+$ where the sum is over all positive
Lyapunov exponents of $\mu$. Suppose that the supremum in the variational principle can be taken over hyperbolic measures. Then we conclude that
$$
h(f) \leq \sup_{p \in Per(f)} \sum \chi_i^+(p)
$$
Using Abdenur,Bonatti and Crovisier's ideas we show it is the case for $C^1-$generic volume preserving diffeomorphisms.

Theorem A will be a consequence of Theorem \ref{new1} and the following theorem.

\begin{theorem}
There exists a residual subset $\mathcal{R}\subset \diff_m^1(M)$ ($M$ of any dimension $d$) such that for any $f \in \mathcal{R}$
$$
h(f) \leq \sup_{p \in Per(f)} \sum_{i=1}^{n_p} \chi^+_i(p)
$$
where the sum is over all positive Lyapunov exponents of the periodic point $p$, counting multiplicity. \label{theoA1}\end{theorem}

Let us see first, how to prove Theorem A.  Take $\mathcal{R}$ as the residual subset derived from the intersection of the ones given by Theorem
\ref{new1} and Theorem \ref{theoA1}. If the supremum in Theorem \ref{theoA1} was taken over hyperbolic periodic points then this sup in
dimension two would be equal $s(f)$ and then Theorem A was proved. In order to overcome this, we divide the proof in two cases. Given a
diffeomorphism $f\in\mathcal{R}$, if $h(f)=0$ then  we have equality by Theorem \ref{new1} ($s(f)\geq 0$), i.e., $h(f)=s(f)$. On the other
hand,  since in dimension two for a periodic point of a conservative diffeomorphism be hyperbolic it's enough  to have positive lyapunov
exponent, if $h(f)>0$ we have that in Theorem \ref{theoA1} the supremum is in fact over the hyperbolic periodic points, and then, we also have
equality between $h(f)$ and $s(f)$, which concludes Theorem A.

In the sequence we prove Theorem \ref{theoA1}.

Using Abdenur, Bonatti and Crovisier \cite{abc}, we can prove the following Proposition.

\begin{proposition}
There is a residual subset $\mathcal{R}\subset \diff_m^1(M)$ such that if $f\in \mathcal{R}$ and $\mu$ is an ergodic measure for $f$, then there
are periodic measures $\mu_{p}$ converging  to $\mu$ in the weak topology, and moreover the vectors formed by the lyapunov exponentes of
$\mu_p$, $L(\mu_p)\in \R^d$, also converge to the Lyapunov vector $L(\mu)\in \R^d$. \label{proA1}\end{proposition}

In fact, they proved this result for dissipative diffeomorphisms, Theorem 3.8 in \cite{abc}. But, unless generic arguments which are also true
in the conservative setting, their theorem is a consequence of Proposition 6.1 there, which we state  here for simplicity.

\begin{proposition}
Let $\mu$ be an ergodic invariant probability measure of a diffeomorphism $f$ of a compact manifold $M$. Fix a $C^1-$neighborhood $\mathcal{U}$
of $f$, a neighborhood $\mathcal{V}$ of $\mu$ in the space of probability measures with the weak topology, a Hausdorff neighborhood
$\mathcal{K}$ of the support of $\mu$, and a neighborhood $\mathcal{O}$ of $L(\mu)$ in $\R^d$. Then there is $g\in \mathcal{U}$ and a periodic
point $p$ of $g$ such that the Dirac measure $\mu_{p}$ associated to $p$ belongs to $\mathcal{V}$, its support belongs to $\mathcal{K}$, and its
Lyapunov vector $L(\mu_p)$ belongs to $\mathcal{O}$. \label{proABC}\end{proposition}

They divided the proof of this proposition in two lemmas, Lemma 6.2 and Lemma 6.3 there. In the first one, given an ergodic measure $\mu$ for
$f$ they found strategic periodic points $p_n$ of some diffeomorphisms $f_n\in \mathcal{U}$ with good properties. Then, in Lemma 6.3 they proved
that in fact the Lyapunov vectors  $L(\mu_{p_n})$ converge to $L(\mu)$.

Hence, since we want this result in the conservative world, given some conservative diffeomorphism $f$ we need to find conservative
diffeomorphisms $f_n$ in any neighborhood of $f$ with same good properties as in their Lemma 6.2. Fortunately, since we have the ergodic closing
lemma, see \cite{Arnaud}, and Frank's lemma, see \cite{LLS} and \cite{Arbieto-catalan}, in the conservative scenario, the proof of a
conservative version of Proposition \ref{proABC} is exactly the same.

Now, using Proposition \ref{proA1} we prove Theorem \ref{theoA1}.

{\it Proof of Theorem \ref{theoA1}:} Let $f \in \mathcal{R}$, where $\mathcal{R}$ is the residual set as in Proposition \ref{proA1}. Given any
$\eps>0$, by variational principle there is an ergodic measure $\mu\in \mathcal{M}(f)$ such that
$$
h(f)< h_{\mu}(f)+\eps.
$$
By Ruelle's inequality $h_{\mu}(f) \leq \sum \chi_i^+(\mu)$, where the sum is over all positive Lyapunov exponents of $\mu$. Now, by Proposition
\ref{proA1} there is a periodic point $p$ of $f$ such that $\sum \chi_i^+(\mu)< \sum \chi_i^+(\mu_p)+\eps$. And then, we have
$$
h(f) < \sup_{p \in Per(f)} \sum_{i}^{+} \chi_i(p)+ 2\eps.
$$
Therefore, since $\eps$ is arbitrarily  small we prove the theorem.

$\hfill\square$

\section{Entropy estimate for symplectomorphisms}

If $f$ is a $C^1$ diffeomorphism over some manifold $M$, and $p$ is a hyperbolic periodic point of $f$, we denote by $H(p,f)$ the set of
transversal homoclinic points of $p$, where  $q\not \in o(p)$ is a {\it transversal homoclinic point} if is a transversal intersection point of
$W^s(o(p),f)$ and $ W^u(o(p),f)$.  If the intersection is not transversal we say that $q$ is a point of {\it homoclinic tangency}.

Zhihong Xia in \cite{Xia} proved that there exists a residual subset $\mathcal{H}\subset \dw$ such that if $f\in \mathcal{H}$ and $p$ is a
hyperbolic periodic point of $f$ then  transversal homoclinic points are dense on stable and unstable manifolds, $W^s(o(p),f)\cup
W^u(o(p),f)\subset cl(H(p,f))$. Let us now state the main proposition and prove Theorems \ref{symplectic} and \ref{symbolic}. We postpone the
proof of this proposition to the next section. Recall $ \chi(p,f)=1/\tau(p,f) \,\log \lambda(p,f) $ is the smallest positive lyapunov exponent
for a hyperbolic periodic point $p$ of $f$.

\begin{proposition} (Main Technical Proposition)
Let $p$ be a hyperbolic periodic point of some non Anosov diffeomorphim $f\in \mathcal{H}$. Given $n>0$ and any neighborhood $\mathcal{N}\subset
\dw$ of $f$,  there exists an open set $\mathcal{U}\subset \mathcal{N}$ such that if $g\in \mathcal{U}$, then $g$ has a basic hyperbolic set
$\Lambda(p(g),n)\subset cl(H(p(g),g))$, where $p(g)$ is the continuation of the hyperbolic periodic point $p$ of $f$ for $g$, such that the
following properties are true

\begin{itemize}
\item[a) ]
$h(g|\Lambda(p(g),n))>\chi(p(g),g)-\frac{1}{n}.$

\item[b)] There exists an ergodic measure $\mu\in \mathcal{M}(\Lambda(p(g),n))$ such that
$$h_{\mu}(g)> \chi(p(g),g)-\frac{1}{n}.$$

\item[c)] For every ergodic measure  $\mu\in\mathcal{M}(\Lambda(p(g),n))$, we have
$$
\rho(\mu,\mu_{p(g)})<\frac{1}{n},
$$
where $\rho$ is a metric which generates the weak topology.
\item[d)] For every periodic point $q\in \Lambda(p(g),n)$, we have
$$\chi(q,g)>\chi(p(g),g)-\frac{1}{n}.
$$
\end{itemize}

\label{proposição princial}\end{proposition}

\subsection{Proof of Theorem \ref{symplectic}}
We denote by $\mathcal{A}$ the set of Anosov diffeomorphisms and consider $\mathcal{D}=\dw-cl(\mathcal{A})$, the complement of the closure of
Anosov diffeomorphisms.

For positive integers  $n$ and $m$, let $B_{n,m}$ be the set of diffeomorphisms  $f$ in $\mathcal{D}$ such that there are $p\in H_n (f)$ and a
hyperbolic basic set  $\Lambda\subset cl(H(p,f)),$ satisfying  $$h(f|\Lambda)>s_n(f)-\dis\frac{1}{m}.$$ Theorem \ref{symplectic}  follows
immediately from the next claim.

{\it {\bf Claim:} $B_{n,m}$ is an open and dense subset of $\mathcal{D}$, for every positive integers $n$ and $m$. }

To prove the claim, we can start with $f\in\mathcal{D}\cap\mathcal{H}$ since we are in a Baire space. Now, let $n,m\in\N$ be anyone.

By definition of $s_n$, there exists $p_0\in H_n(f)$ such that
$$s_n(f)=\chi(p_0,f).$$

Using Proposition \ref{proposição princial}, we can find $f_1$  $C^1-$close to $f$ such that $f_1$ has a hyperbolic basic set $\Lambda\subset
cl(H(p_0(f_1),f_1))$, where $p_0(f_1)$ is the continuation of $p_0$ for $f_1$, and
\begin{align}
h(f_1|\Lambda)&> \chi(p_0(f_1),f_1)-\frac{1}{3m}.\label{e3}
\end{align}
Now, using the robustness of $\Lambda$ and $p_0$, the invariance of topological entropy  
and that $s_n$ is continuous, we have the following for $g$ $C^1-$near $f_1$
\begin{align*}
h(g)&\geq h(g|\Lambda(g))
\\
&= h(f_1|\Lambda)
\\
&>\chi(p_0(f_1),f_1)-\frac{1}{3m}
\\
&\geq\chi(p_0,f)-\frac{2}{3m}
\\
&= s_n(f)-\frac{2}{3m}
\\
&>s_n(g)-\frac{1}{m},
\end{align*}
which proves the claim and then  Theorem B.

$\hfill\square$

\vs

\subsection{Proof of Theorem \ref{symbolic}.}

 Remember that $(Y,\sigma)$ is a symbolic extension of $(M,f)$ if there exists a continuous surjective map $\pi:Y\rightarrow M$ such that
$\pi\circ \sigma=f\circ \pi$. As we already comment, it may happen that symbolic extensions of a system have larger entropy and carry much more
information than the system.

Hence, let $$ h_{ext}^{\pi}(\mu)=\sup\{h_{\nu}(\sigma|Y):\pi_*\nu=\mu\}, \, \text{ for } \mu\in\mathcal{M}(f),$$ and observe that principal
symbolic extensions minimize these functions.

Let $S(f)$ be the set of all possible symbolic extensions  $(Y,\sigma,\pi)$ of $(M,f)$. We say that $S(f)=\emptyset$ if there is no symbolic
extension of $(M,f)$. We define the {\it residual entropy of the system} by
$$
h_{res}(f)=h_{sex}(f)-h(f),
$$
where
$$
h_{sex}(f)=\left\{\begin{array}{ll} \inf\{h_{ext}^{\pi}(\mu):(Y,\sigma,\pi)\in S(f)\}& \text{ if } S(f)\neq \emptyset
\\
\infty& \text{ if } S(f)=\emptyset
\end{array}\right.
$$

In terms of this, to prove Theorem \ref{symbolic} we need to show that $h_{sex}(f)=\infty$ for all non Anosov diffeomorphism $f$ in some
residual subset $\mathcal{B}\subset \dw$.

Given $f:M\rightarrow M$ a homeomorphism in a compact metric space $M$, an increasing sequence $\alpha_1\leq \alpha_2\leq\ldots$ of partitions of
$M$ is called  {\it essential} for $f$ if
\begin{itemize}
\item [1.] $diam(\alpha_k)\rightarrow 0$ when $k\rightarrow \infty$, and
\item [2.] $\mu(\partial\alpha_k)=0$ for every $\mu\in\mathcal{M}(f)$. Where $\partial\alpha_k$
denotes the union of boun\-da\-ries of all elements of the partition  $\alpha_k$.
\end{itemize}

A {\it sequence of simplicial partitions} is a nested sequence $\mathcal{S}=\{\alpha_1,\alpha_2,\ldots\}$ of partitions whose diameters go to
zero, and each $\alpha_k$ is given by some smooth triangulation of $M$. By Proposition 4.1 in \cite{Newhouse3} there is a residual subset
$\mathcal{R}_{\mathcal{S}}\subset \dw$ such that if $f\in \mathcal{R}_{\mathcal{S}}$ then $\mathcal{S}$ is an essential sequence of partitions
for $f$.

Hence, for every $k$ fixed, the function $$h_k(\mu)=h_{\mu}(\alpha_k),$$ is the infimum of continuous functions  over $\mathcal{M}(f)$, and then
is upper semicontinuous. Here $h_{\mu}(\alpha_k)$ is the entropy of the partition $\alpha_k$ for $f$. The following proposition gives us a very
useful way to prove non existence of symbolic extensions.  It was also proved in \cite{Newhouse3}.

\begin{proposition}
Let $f\in\mathcal{R}_{\mathcal{S}}$ and suppose $\mathcal{E}$ be some compact subset in $\mathcal{M}(f)$ such that there exists a positive real
number $\rho_0$  such that each $\mu\in\mathcal{E}$ and $k>0$,
$$
\limsup_{\nu\in\mathcal{E},\nu\rightarrow \mu}[h_{\nu}(f)-h_k(\nu)]>\rho_0.
$$
Then,
$$
h_{sex}(f)=\infty.
$$
\label{prop prin}\end{proposition}

Recall $H_n(f)$ denotes the set of hyperbolic periodic points with period smaller or equal than $n$, and let $H(f)=\cup_n H_n(f)$. By Pugh's
closing lemma the set of diffeomorphisms $\mathcal{R}_1$ formed by $f$ with $H(f)\neq\emptyset$ is open and dense in $\dw$. Hence, it's well
defined $\tau(f)$ as the smallest period of the elements in $H(f)$ for every $f\in\mathcal{R}_1$, and then let $\mathcal{R}_{1,m}\subset
\mathcal{R}_1$ be the set of diffeomorphisms $f$ with $\tau(f)=m$. Note, $\mathcal{R}_1$ is a disjoint union of $\mathcal{R}_{1,m}$.

Now, for each $f\in \mathcal{R}_1$ we define
$$
\chi(f)=\sup\{\chi(p,f):\, p\in H(f)\text{ and } \tau(p,f)=\tau(f)\}.
$$
Then, $\chi(f)>0$ and depends continuously on $f\in\mathcal{R}_1$.

Recalling that $\mathcal{A}\subset \dw$ is the set of Anosov diffeormophisms, let $\mathcal{R}_{2,m}=\mathcal{R}_{1,m} \backslash
cl(\mathcal{A}), $ which implies $ \mathcal{R}_1\backslash cl(\mathcal{A})=\bigcup_{m} \mathcal{R}_{2,m}. $

Suppose now that $\Lambda$ is an f-invariant periodic set  with basis  $\Lambda_1$ and $\alpha={A_1,A_2,...,A_s}$ some finite partition of $M$.
We say that  $\Lambda$ is {\it subordinate} to $\alpha$ if for each  positive integer $n$, there exists an element $A_{i_n}\in\alpha$ such that
$f^{n}(\Lambda_1)\subset A_{i_n}$. Hence, if $\mu\in \mathcal{M}(f|\Lambda)$ then $h_{\mu}(\alpha)=0$.

Now, given a positive integer $n$,  we say that a diffeomorphism $f$ satisfies  property  $\mathcal{S}_n$ if for every $p\in H_n(f)$ with
$\chi(p,f)>\dis\frac{\chi(f)}{2}$, {\it
\begin{itemize}
\item [1.] There exists a hyperbolic basic set of zero dimension $\Lambda(p,n)$ for $f$ such that
\begin{equation*}
\Lambda(p,n)\cap\partial\alpha_n=\emptyset \; \text{ and } \; \Lambda(p,n)\, \text{ is subordinate to } \,\alpha_n. \label{cond
1}\end{equation*}

\item[3.] There exists an ergodic measure $\mu\in \mathcal{M}(\Lambda(p,n))$ such that
\begin{equation*}
h_{\mu}(f)> \chi(p,f)-\frac{1}{n}. \label{cond 3}\end{equation*}

\item[4.] For every ergodic measure  $\mu\in\mathcal{M}(\Lambda(p,n))$, we have
\begin{equation*}
\rho(\mu,\mu_p)<\frac{1}{n}. \label{cond 4}\end{equation*}

\item[5.] For every periodic point $q\in \Lambda(p,n)$, we have
\begin{equation*}
\chi(q,f)>\chi(p,f)-\frac{1}{n}. \label{cond 5}\end{equation*}
\end{itemize}}

Given positive integers  $m\leq n$, let $\mathcal{D}_{m,n}\subset \mathcal{R}_{2,m}$ be the subset of  diffeomorphisms  $f$ satisfying property
$\mathcal{S}_n$.

Since periodic points in $H_n(f)$ with smallest positive lyapunov exponent bigger than $\chi(f)/2$ are finite, directly  from Proposition
\ref{proposição princial}, conditions (3), (4) and (5) above are satisfied for diffeomorphisms in an open and dense subset of
$\mathcal{R}_{2,m}$. Now, fixed some partition $\alpha_n$ we can take a smaller open set $U$  where we build the hyperbolic set $\Lambda(p,n)$,
as we can see in the proof of Proposition \ref{proposição princial} in the next section, in order to obtain that the set
$\Lambda(p,n)$ 
is subordinate to $\alpha_n$. Therefore, since this is a robust property we have proved the following lemma.

\begin{lemma}
For positive integers $m\leq n$, $\mathcal{D}_{m,n}$ is open and dense in $\mathcal{R}_{2,m}$. \label{lema principal}\end{lemma}

Now, using property $S_n$ and the above lemma the proof of Theorem \ref{symbolic} is similar to the proof of Theorem 1.3 in \cite{Newhouse3},
but for convenience we reproduce it again.

{\it Proof of Theorem \ref{symbolic}:} Let
$$
\mathcal{R}_2= \bigcup_{m\geq 1}\bigcap_{n\geq m} D_{m,n}.
$$
By Lemma \ref{lema principal}, we have that $\mathcal{R}=\mathcal{R}_{\mathcal{S}}\cap(\mathcal{R}_2\cup \mathcal{A})$ is a residual set in
$\dw$.

What we show now is that all non Anosov diffeomorphism $f\in \mathcal{R}$ has no symbolic extension, i.e., $h_{sex}(f)=\infty$.

Let $f\in \mathcal{R}$ be a non Anosov diffeomorphism. Now, we define
$$
\mathcal{E}_1=\left\{\mu_p; \,\text{ such that }\, p\in H(f) \,\text{ and }\, \chi(p,f)>\frac{\chi(f)}{2}\right\},
$$
and let $\mathcal{E}$ denotes its closure in $\mathcal{M}(f)$.

Using this set and  property $\mathcal{S}_n$, we show that the hypothesis of Proposition  \ref{prop prin} are satisfied if we take
$\rho_0=\dis\frac{\chi(f)}{2}$. For this, it's enough to verify the hypothesis for every $\mu_p\in \mathcal{E}_1$, and fixed $k\in\N$.

Hence, given $n\in\N$ big enough, since $f\in \mathcal{R}$, there exists a hyperbolic basic set $\Lambda(p,n)$ for $f$ subordinate to
$\alpha_n$, and an ergodic measure $\nu_n\in\mathcal{M}(\Lambda(p,n))$ such that $\rho(\nu_n,\mu_p)<1/n$ and
\begin{equation}
h_{\nu_n}(f)>\chi(p,f)-\frac{1}{n}. \label{d1}\end{equation}

Since we can suppose  $n>k$, we have that $\alpha_n$ is smaller than $\alpha_k$ and then  $\Lambda(p,n)$ is also subordinate to $\alpha_k$.
Hence, as $\nu_n\in \mathcal{M}(\Lambda(p,n))$
\begin{equation}
h_k(\nu_n)=0. \label{d2}\end{equation}

Therefore, we have that  $\nu_n\rightarrow \mu_p$, when $n\rightarrow\infty$, and
$$
|h_{\nu_n}(f)-h_k(\nu_n)|=h_{\nu_n}(f)>\chi(p,f)-\frac{1}{n}>\rho_0,
$$
where the last inequality is satisfied for big values of $n$, since $\mu_p\in \mathcal{E}_1$.

To complete the proof we need to show that $\nu_n$ is in  $\mathcal{E}$, for every $n$. To see this, we use that $\nu_n$ is approximated
by periodic measures since is an ergodic measure supported in a hyperbolic basic set. 
That is, there exist  $q_{m,n}\in\Lambda(p,n)$, hyperbolic periodic points of $f$, such that $\mu_{q_{m,n}}$ converges to $\nu_n$ in the weak
topology. This way, our work is reduced to show that $\mu_{q_{m,n}}\in\mathcal{E}_1$, which is direct from item 5 of property $\mathcal{S}_{n}$,
provided $f\in \mathcal{R}$. The proof of theorem \ref{symbolic} is complete.

$\hfill\square$

\section{ Symplectic Perturbations: proof of Proposition \ref{proposição princial} }
Before going into the proof of the proposition let us recall some basic fact about symplectic structure. Let $(V, \omega)$ be a symplectic
vector space of dimension $2n$. For any subspace $W \subset V$ its {\it symplectic orthogonal} is defined as
$$
  W^{\omega} = \{  v \in V; \omega(v, w) = 0 \quad \text{for all} \quad w \in W \}.
$$
The subspace $W$ is called {\it symplectic} if $W^{\omega} \cap W = \{0\}.$  $W$ is called {\it isotropic} if $W \subset W^{\omega},$ that is
$\omega|W\times W =0.$ A special case of isotropic subspace is a {\it Lagrangian} subspace, i.e., when $W = W^{\omega}.$ For a symplectic
manifold $(M, \omega)$ and a symplectic diffeomorphism $f$ it is easy to see that for any point on an unstable (stable) manifold of a hyperbolic
periodic point, the tangent space to  unstable (stable) manifold is a Lagrangian subspace.

The proof of main proposition is done in three steps where the second and third are the main ones and use the symplectic structures.

Let $f\in\mathcal{H}$ be a non Anosov diffeomorphisms.

\begin{itemize}
\item [Step 1-] We find $g_1$ $C^1-$close to $f$ such that $p$ is still a hyperbolic periodic point of $g_1$,
and $g_1$ exhibits one homoclinic tangency between $W^s(o(p), g_1)$ and $W^u(o(p), g_1)$. Moreover, $g_1=Df_p$ in a small neighborhood of the
orbit of $p$ (in local symplectic coordinates).

\item [Step 2-] We find $g_2$ $C^1-$ close to $g_1$ where $g_2$ admits a segment of line of homoclinic tangency. We should perform perturbations in the symplectic high dimensional setting.

\item [Step 3-] Finally, we  perturb $g_2$ to obtain $g$ with a hyperbolic invariant set satisfying  the properties required by the proposition.
All $C^1$-perturbations of $g$ also share the same property for the corresponding hyperbolic set.
\end{itemize}

{\it Proof of step 1:} The way, to create homoclinic tangency is in the lines of the proof of Newhouse (step 6, Theorem 1.1 in
\cite{Newhouse2}). After that, we use a pasting lemma of Arbieto-Matheus \cite{pasting lema} and continuity of compact parts of stable and
unstable manifolds to obtain a tangency and linearization in a neighbourhood of the periodic point. We should point out that because of high
dimensions of stable and unstable manifolds, by homoclinic tangency we obtain at least one (it can be unique) commom direction between the
tangent spaces of these manifolds at the point of tangency.

\vsm
 {\it Proof of step 2:}
For simplicity we suppose $p$ is a hyperbolic fixed point of $g_1$, 
and let $V$ be a neighborhood of $p$ where in local symplectic coordinates $g_1$ is linear, with $E_p^s=\R^{n}\times \{0\}^n$ and
$E_p^u=\{0\}^n\times\R^{n}$.
 Moreover, by Darboux's Theorem, we can also suppose in $V$ that
$\omega$ is the standard 2-form for $\R^{2n}$, $\omega=\sum dx_i\wedge dy_i$.

Let $q$ be the point of homoclinic tangency between $W^s_{loc}(p,g_1)$ and $W^u(p,g_1)$, such that $q\in V$ and $g_1^{-1}(q)\not\in V$. Hence,
we can take one small neighborhood $U\subset V$ of $q$ such that $g_1^{-1}(U)\cap V=\varnothing$. We denote by $D$ the connected component of
$W^u(p,g_1)\cap U$ that contains $q$.

We want now to perturb $g_1$ in order to get an interval of homoclinic tangency. Since stable (unstable) manifolds is a graphic, it's not
difficult to do this in the conservative scenario using the point of tangency $q$. In symplectic case this may be done using the fact that
stable (unstable) manifold is a lagrangian manifold as we explain it below.

First, we consider another simplectic coordinate on $U$ in order to simplify the notation such that $q$ is the origem, and we have the following
$$
W^s_{loc}(p,g_1)\cap U=\{y_1=y_2=...=y_n=0\}\cap U,
$$
$$
T_{q}D=\{y_{1}=x_2=...=x_n=0\},
$$
and so
$$
W^s_{loc}(p,g_1)\cap U\cap T_{q} D=\{e_1\},
$$
where we are considering $\{e_1,...,e_n,....,e_{2n}\}$ as the canonical basis of $\R^{2n}$. Note we are using that
$dim(T_qW^s(p,g_1)+T_qW^u(p,g_1))=2n-1$, which we can suppose after some perturbation if necessary.

The following lemma is the technical point that allows us to build the interval of homoclinic tangency for symplectomorphisms.

\begin{lemma}
There exists a symplectic diffeomorphism $\phi:U\rightarrow \R^{2n}$ over its image, $C^1$ close to identity map $Id$ in a small neighborhood of
$q$, such that $\phi(D)\cap W^s_{loc}(p,g_1)\cap U$ contains one segment of line. \label{curva}\end{lemma}

\proof

Just here we use coordinates $(x,y)$ with respect to the following decomposition of the space $\R^{2n}=E\oplus F$, where $E$ and $F$ are
generated by $\{e_1,e_{n+2},...,e_{2n}\}$ and $\{e_2,...,e_{n+1}\}$, respectively. Recall that $E=T_q D$, and $q=(0,0)$ by the choice of the
coordinate.

Now, since $D$ is locally a graphic of one function with the same class of differentiability that $g_1$, there exists a $C^1$ map
$j:B\subset\R^{n}\rightarrow \R^{2n}$, $j(x)=(x,r(x))$, such that $j(B)=D$. Moreover, $j$ is such that $Dr(0)=0$, and since $D\subset
W^u(p,g_1)$ is a lagragian submanifold, we have $j^*\omega=0$,  where $j^*\omega$ is the pull-back of the form $\omega$ by $j$. Analogously, if
$i:\R^{n}\rightarrow \R^{2n}$ is the natural inclusion, $i(x)=(x,0)$,  we have $i^*\omega=0$ (recall $\omega$ in $U$ is the standard 2-form on
$\R^{2n}$).

Let us define $\phi:U\rightarrow \R^{2n}$ by $\phi(x,y)=(x,y-r(x))$. Taking $U$ smaller, if necessary, $\phi$ is in fact a diffeomorphism from
$U$ into its image and $C^1$ near $Id$, since $Dr(0)=0$. Hence, to conclude the lemma we need to show that $\phi$ is indeed symplectic. Denoting
the projection in the first coordinate by $\pi:\R^{2n}\rightarrow \R^{n}$, $\pi(x,y)=x$, we can rewrite $\phi$ in the following way $
\phi=Id+i\circ\pi-j\circ\pi. $ Then,
$$\phi^*\omega=\omega+\pi^*i^*\omega-\pi^*j^*\omega=\omega,$$ where we use that $i^*\omega=j^*\omega=0$ in the second equality.
Therefore, the lemma is proved.
$\hfill\square$

\vsm

Using the pasting lemma of Arbieto-Matheus \cite{pasting lema} in symplectic case and the map $\phi$ given by Lemma \ref{curva} we can find
$R:U\rightarrow U$ $C^1$ close to identity $Id$, with $R=\phi$ in a small neighborhood of $q$, and $R=Id$ outside another small neighborhood
containing the last one. Hence, considering $\tilde{R}:M\rightarrow M$ with $\tilde{R}=Id$ in $U^{c}$ and $\tilde{R}=R$ in $U$, and taking
$g_2=\tilde{R}\circ g_1$  we have a $C^1$ perturbation of $g_1$ that coincides with $g_1$ in $(g_1^{-1}(U))^{c}$. Moreover, the most important
is that this perturbation exhibits an interval of homoclinic tangency as we wanted. More precisely, there is one segment of line $I\subset
W^s_{loc}(p,g_2)\cap W^u(p,g_2)\cap U$. Note that $I$ is in the space generated by the unit vector $e_1$, and after some symplectic coordinate
changed inside $U$, we can suppose $I\subset\{(x_1,0,...,0), \, -2a\leq x_1\leq 2a\}$, for some $a>0$ small enough and usual coordinates of
$\R^{2n}$.

\vsm {\it Proof of Step 3:} The idea now is to use this interval of tangency to create hyperbolic sets with the properties required by the
proposition. Let $N$ be a big positive integer and $\delta>0$ an arbitrary small real number. As before (using pasting lemma) we can find a
symplectic diffeomorphism $\Theta:M\rightarrow M$, $\delta-C^1$ near $Id$, $\Theta=Id$ in $U^c$ and
$$
\Theta(x,y)=\left(x_1,...,x_n,\,y_1+A\cos\frac{\pi x_1 N}{2a},\, y_2,...,y_n\right), \text{ for } (x,y)\in B(0,r)\subset U,
$$
for $A=\dis\frac{2Ka\delta }{\pi N}$ and $r>0$ small enough, where $K$ is some constant depending only on the symplectic coordinate on $U$.
Hence,  $g=\Theta\circ g_2$ is $\delta-C^1$ close to $g_2$ and moreover $g=g_2$ in the complement of $g_2^{-1}(U)$.  Although the diffeomorphism
$g$ depends on $N$, we always denote these diffeomorphisms by $g$. We would like to note this perturbation is an adaptation of Newhouse's snake
perturbation for higher dimensions, i.e., it destroys the interval of tangency and creates $N$ transversal homoclinic points for $p$ inside $U$.

Using the function $\Theta$ we choose two good points on unstable manifold of $p$ for $g$, $z_1= \Theta(-a,0, \cdots, 0)$ and $z_2= \Theta(a,0,
\cdots, 0)$. Now, we consider  $\gamma_1$ and $\gamma_2$ two transversal disks to unstable manifold $W^u(p,g)$ at $z_1$ and $z_2$, respectively.

From now on we use the symplectic coordinate on $V$ fixed before. Note that $g$ is equal to $g_1$ inside $V$ and so $g$ is linear in $V$.

Given a set  $E$, we denote by $C(E,x)$ the connected component of $E$ containing $x$. By $\lambda-$Lema  and choice of $\gamma_1$ and
$\gamma_2$, $C(g^{-j}(\gamma_1)\cap V,g^{-j}(z_1))$ and $C(g^{-j}(\gamma_2)\cap V,g^{-j}(z_2))$ accumulate on  $W^s_{loc}(p,g)$ for big values
of $j>0$.

Hence, if $D^s=W^s_{loc}(p,g)\cap U$ then for $j$ big enough we can define the rectangle $D_j=D^s \times D^u_j $ as being the cartesian
product between $D^s$ 
and $D^u_j$, where $D^u_j$ is the smallest possible disk in $\{(0,\ldots,0,y_1,\ldots,y_n),\, y_i\in\R\}$ such that
$\pi_2(C(g^{-j}(\gamma_i)\cap V,g^{-j}(z_i)))\subset D^u_j$, for $i=1,2$. Here $\pi_2(x,y)=y$ stands for the projection on the second
$n$th-coordinates of $\R^{2n}$, and recall we are considering $V$ inside Euclidean space with $E_p^s=\R^n\times \{0\}^n$ and $
E_p^u=\{0\}^n\times \R^n$.

Let $J\subset U$ be some small enough disk inside the unstable manifold $W^{u}(p,g)$ containing the $N$ transversal homoclinic points built
before, and let $T>>0$  such that $g^{-T}(J)\subset V$, and moreover $g^{-T}(\gamma_i)$, $i=1,2$, is close to $W^s_{loc}(p,g)$. We denote by
$\tilde{\Gamma}$ the $A/2$-neighborhood of $J$, and define $\Gamma=g^{-T}(\tilde{\Gamma})$.

Now, let $t_0$ be the smallest positive integer such that $C(g^{-t_0}(\gamma_i),g^{-t_0}(z_i))$  is  $A/2-C^1$ close to $W^s_{loc}(p,g)$,
$i=1,2$. Note that if  $t'\geq t_0$, and $g^{t'-T}(D_{t'})\subset \Gamma$, then $g^{t'}(D_{t'})\cap(D_{t'})$ contains $N$ disjoint connected
components. Hence, we consider  $z_3= (b,0,...,0)$ and $z_4=(b',0,...,0)$ two points on local stable manifold of $p$, where $b$ and $b'$ are the
left and right boundary points in the first coordinate of $W^s_{loc}(p,g)\cap U$. Also, let $\gamma_3$ and $\gamma_4$ be two transversal disks
to $W^s_{loc}(p,g)$ at $z_3$ and $z_4$, respectively. By $\lambda-$lemma again we can define $t_1$ as the smallest possible positive integer
such that
$$
C(g^{t_1}(\gamma_i),g^{t_1}(z_i))\cap C(g^{-T}(\gamma_j),g^{-T}(z_j))\cap \Gamma\neq \emptyset, \text{ for } j=1,2 \text{ and } i=3,4.
$$

Finally, we define $t=\max\{t_0,\, t_1+T\}$ , see figure 1. Note $t$ depends on $N$ since $t_0$ and $t_1$ depends, and also observe that $t$
goes to infinity when $N$ goes.

\begin{figure}[httb]
\vsm\centering
\includegraphics[scale=1.5]{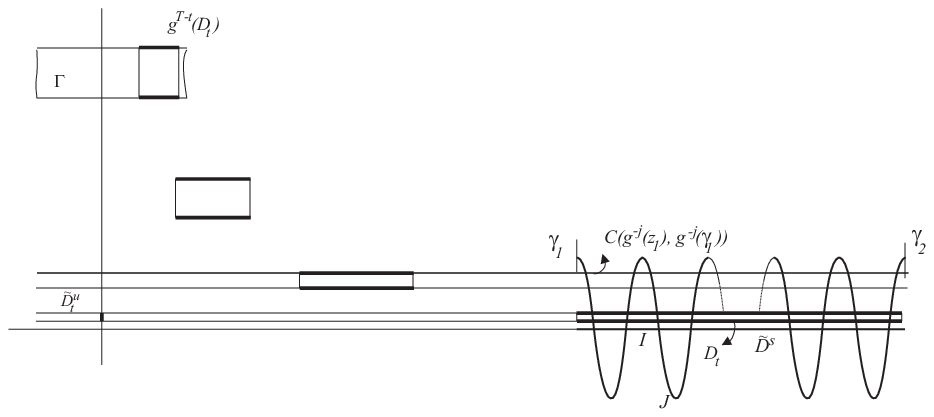}
\caption{}\vsm\label{figura5}\end{figure}

By the comments above and choice of $t$, we have that $g^{t}(D_{t})\cap D_{t}$ has $N$ disjoint connected components, and $t$ is the smallest
possible one such that $D_t$ is $A/2-C^1$ close to $W^s_{loc}(p,g)$ and $g^t(D_t)$ is $A/2-C^1$ close to $J\subset W^u(p,g)$. Therefore, since
we have a horseshoe with N legs, the maximal invariant set in $D_t$ for $g^t$
$$
\tilde{\Lambda}(p,N)=\bigcap_{j\in\Z} g^{tj}(D_t)
$$
is a hyperbolic set with dynamics conjugated to a shift of $N$ symbols. Then  $h(g^{t}|\tilde{\Lambda}(p,N))=\log N$, and taking
$$\Lambda(p,N)=\bigcup_{j=1}^{t} g^j(\tilde{\Lambda}(p,N))$$ we have $h(g|\Lambda(p,N))=\dis\frac{1}{t}\log N$.

The following lemma is the main point in this step.

\begin{lemma}
For $A$ and $t$ defined as before, there exists a positive integer $K_1$ independent of $A$, such that $$ A<K_1 \min \{ \|Dg_p^{-t}|E^u\|,\,
\|Dg_p^{t}|E^s\|\}.$$ \label{afirma}\end{lemma}

\proof Since $V$ is a neighborhood of $p$ where $g$ is linear, if $m$ is the biggest one such that $g^j(x)\in V$ for $0\leq j\leq m$, there
exist constants $K_2$ and $K_3$ depending on the symplectic coordinate on $V$  such that

\begin{equation}
K_2 \|Dg_p^{m}|E^u\|^{-1}\leq d(x,W^s_{loc}(p,g))\leq K_3 \|Dg_p^{-m}|E^u\|, \label{desigualdade 1}\end{equation} for $x\in V$. Analogously, if
$m$ is the biggest one such that $g^{-j}(x)\in V$ for $0\leq j\leq m$, then  there exist constants  $K_4$ and $K_5$ such that
\begin{equation}
K_4 \|Dg_p^{-m}|E^s\|^{-1} \leq d(x,W^u_{loc}(p,g))\leq K_5 \|Dg_p^{m}|E^s\|. \label{desigualdade 2}\end{equation}

Now, by  choice of $t$, either there exists $z\in D_t$ such that
\begin{equation}
d(g(z), W^s_{loc}(p,g))\geq A/2, \label{3.5}\end{equation} or there exists  $z\in g^t(D_t)$ such that
\begin{equation}
d(g^{-1}(z), J)\geq A/2. \label{3.6}\end{equation} Suppose the first case. Recall that for $j>T$ the rectangle  $D_j$ is defined and moreover
$D_j\subset V$, which implies $g(z),\,g(g(z)),...,\,g^{t-1-T}(g(z))\in V$. Hence, using  inequality (\ref{desigualdade 1}) we have
$$
\dis\frac{A}{2}\leq K_3 \|Dg_p^{-t+T+1}|E^u\|.
$$
On the other hand, using  inequality (\ref{desigualdade 2}) and the neighborhood $\Gamma$,  we can do the same thing for the second case,
obtaining
$$
\dis\frac{A}{2}\leq K_5 \|Dg_p^{t-1-T}|E^s\|.
$$
And then, since $Dg$ is bounded and $T$ is independent of $A$ we can find $K_1$ as we claimed.

$\hfill\square$

From now on fix $n$ a large positive integer as Proposition \ref{proposição princial} requires.

Since $A=\displaystyle\frac{2Ka\delta}{\pi N}$,  using Lemma \ref{afirma} and for  $N$ big enough, we have
$$
\dis\frac{1}{t}\log N> \min\left\{\frac{1}{t}\log\|Dg_p^{-t}|E^u\|^{-1},\, \frac{1}{t}\log\|Dg_p^{t}|E^s\|^{-1} \right\} -\frac{1}{2n}.
$$
Observe now, when $t$ goes to infinity the above minimum converges to the minimum between the smallest positive Lyapunov exponent, $\chi(p,g)$
as defined before,  and the absolute value of the biggest negative Lyapunov exponent of $p$ for $g$. Moreover, since we are in the symplectic
scenario these two numbers are equal. Therefore, provided $t$ goes to infinity when $N$ goes, we can find a positive integer $N_1$ such that
$$
\dis\frac{1}{t}\log N_1> \chi(p,g) -\frac{1}{n}.
$$
Which implies that it's possible to find some $C^1-$perturbation $g$ of $f$ such that
$$
h(g|\Lambda(p,N_1))> \chi(p,g)-\frac{1}{n}.
$$

For general case, when $p$ is not a fixed point of $g_1$, i.e., $\tau(p,g_1)>1$, we have that  $q\in W_{loc}^s(p,g_1)\cap W^u(f^j(p),g_1)$, for
some $0\leq j<\tau(p,g_1)$. Then as we did before,  we can find some perturbation $g$ of $g_1$ and $t=\tau(p,g)\tilde{t}+j$ such that $g^t$ has
a hyperbolic basic set $\tilde{\Lambda}(p,N)$. Moreover, there is a relation between the norm of $Dg^{\tau(p,g)}$ and $A$
as in the Lemma \ref{afirma}, changing $t$ by  $\tilde{t}$. %
 Hence, we can find $N_1$ such that

\begin{equation}
h(g|\Lambda(p,N))> \chi (p,g)-\frac{1}{n}, \text{ for } N\geq N_1. \label{item a}\end{equation}

Now, since $\Lambda(p,N)$ is conjugated to the product between some finite permutation dynamics and the shift of $N$ symbols, there exists an
ergodic measure $\mu_N\in\mathcal{M}(\Lambda(p,N))$ that maximizes the topological entropy. Hence, directly from (\ref{item a})
\begin{equation}
h_{\mu_N}(g)> \chi (p,g)-\frac{1}{n}, \text{ for } N\geq N_1.\label{item b}\end{equation}

We suppose from now on that $p$ is fixed, being the general case similar deduced as we did before. Next, we find a positive integer $N_2$ such
that if $\mu\in \mathcal{M}(f|\Lambda(g,N_2))$ is ergodic then $\rho(\mu,\mu_p)<1/n$  as required. For this, given $\zeta>0$ arbitrary small
it's enough to find $N=N(\zeta)$ such that  (orbit of) any point of $\Lambda(p,N)$ visits very frequently the ball of radius $\zeta$ and center
$p$.

Provided $p$ is a hyperbolic fixed point we have
$$
\bigcap_{i\in\Z}g^{i}(V)=\{p\}.
$$
Hence, given $\zeta>0$ arbitrary small, there exists a positive integer $n_1\geq T$, depending on $\zeta$, such that for every  $n_2\geq n_1$
$$
diam \, \bigcap_{-n_2\leq i \leq n_2} g^{i}(V)<\zeta.
$$

Now, if $\overline{V}=\bigcap_{i=0}^{ln_1}g^{-i}(V)$ and $z\in\overline{V}$, then for every $r\in[n_1,(l-1)n_1)$ we have that
$$
g^{r}(z)\in\bigcap_{|i|<n_1}g^{i}(V)\subset B_{\zeta}(p).
$$
So, the fraction of time in  $[0,ln_1)$ that the orbit of $z$ stay in $B_{\zeta}(p)$ is $\dis\frac{l-2}{l}$.

Recall that $t$ is the period of the periodic set $\Lambda(p,N)$ of $g$, and let us define $k=t-T$. Given $N$ big enough, let $l\in\N$ be such
that $(l+1)n_1\geq k> ln_1$. Since for every $z\in \Lambda(p,N)$ there exists $r\in[0,t)$ such that  $g^r(z)\in\overline{V}$, the frequency of
the orbit of  $z$ passing in $B_{\zeta}(p)$ is bigger than
$$
\frac{(l-2)n_1}{(l+1)n_1+T}.$$

Provided $l\rightarrow\infty$ when $N\rightarrow\infty$, given $\zeta_1>0$ we can choose  $N_2$ such that the frequency of the orbit of every
$z\in\Lambda(p,N)$ passing in $B_{\zeta}(p)$ is bigger than $1-\zeta_1$, and then choosing $\zeta_1$ smaller if necessary we have

\begin{equation}
d(\mu,\mu_p)<\frac{1}{n}, \text{ for every ergodic measure } \mu\in\mathcal{M}(\Lambda(p,N)), \; N\geq N_2.
\end{equation}

Finally, we find $N_3$ in order to obtain property (d) for $\Lambda(p,N)$, $N\geq N_3$.

We define
$$
V_k^{u}=V\cap g(V)\cap...\cap g^{k}(V), \text{ and }
$$
$$
V_k^{s}=V\cap g^{-1}(V)\cap...\cap g^{-k}(V).
$$

Given vectors $v,w\in\R^{2n}$ and subspaces  $E,F\subset\R^{2n}$ we define
$$
ang(v,w)=\left|\tan\left[\arccos\left(\frac{<v,w>}{\|v\|\|w\|}\right)\right]\right|,$$

$$
ang(v,E)=\min_{w\in E,\, |w|=1}\, ang(v,w)\quad \text{and}\quad ang(E,F)=\min_{w\in E,\, |w|=1} \, ang(w,F).
$$

\begin{remark}
Another definition of the angle between two subspaces in literature is the following: If $\R^n=E\oplus F$ is some decomposition, let $L:
E^{\perp}\rightarrow E$ be the linear map such that  $F=\{w+Lw; \; w\in E^{\perp}\}$, and then some authors define the angle between $E$ and $F$
as $\|L\|^{-1}$. Nevertheless, there is an equivalence between this definition and the one presented here.
\end{remark}

We need the following lemma.

\begin{lemma}
With above definitions, there exists constant  $K_6$ such that if $z\in V^{s}_k$, $v\in \R^{2n}\backslash E_p^s$ and $ang(g^{k}(v),E_p^s)\geq
1$, then
$$
|Dg^{k}(z)(v)|\geq K_6 \|Dg_p^{-k}\|^{-1}|v| \,\min \{ang(v,E_p^{s}),\, 1\}.
$$

\label{lema}\end{lemma}

\proof

Using the decomposition of $\R^{2n}$ fixed on $V$, we have $v=(v^s,v^u)$, $v^{s(u)}\in E_p^{s(u)}$, for every $v\in \R^{2n}$. Let
$|v|'=\max\{|v^s|,|v^u|\}$ be the maximum norm.

Since ${E_p^s}^{\perp}=E_p^u$, and $Dg^{k}(z)=Dg_p^k$ if $z\in V$, we have
\begin{equation}
ang(v,E_p^s)=\frac{|v^u|}{|v^s|} \; \text{ and } \; 1\leq ang(Dg^k(z)(v),E_p^s)=\frac{|Dg_p^k(v^u)|}{|Dg_p^k(v^s)|}. \label{equa
1}\end{equation}

Then,

\begin{align*}
|Dg^k(z)(v)|'&= |Dg_p^k(v^u)|
\\
&\geq \|Dg^{-k}|E_p^u\|^{-1}|v^u|,
\\
&= \|Dg_p^{-k}|E_p^u\|^{-1}|v^s|\, ang(v,E^s);
\end{align*}
which implies
\begin{equation}
|Dg^k(z)(v)|'\geq \|Dg_p^{-k}|E^u\|^{-1}|v|'\, \min\{ ang(v,E_p^s),\, 1\}. \label{eq19}\end{equation}

Therefore, by the equivalence between norms, the result follows.

$\hfill\square$

Recall now, $$\Lambda(p,N)=\bigcup_{i=0}^{t-1} g^i(\tilde{\Lambda}(p,N))$$ is a hyperbolic set for $g$, with $\tilde{\Lambda}(p,N)\subset V$,
and $t=k+T$ where
 $g^{i}(\tilde{\Lambda}(p,N))\subset V$ for $0\leq i\leq k$. Moreover, by construction of
$\tilde{\Lambda}(p,N))$ we know that the hyperbolic decomposition $T_{\Lambda(p,N)}M=\tilde{E}^s\oplus \tilde{E}^u$ is such that
$\tilde{E}^s(z)$ and $\tilde{E}^u(g^{k}(z))$ are close to $E^s(p)$ and $E^u(p)$, respectively, for every $z\in \tilde{\Lambda}(p,N)$. In
particular, $ang(Dg^{k}(z)(v),E^s(p))>1$ for $v\in \tilde{E}^u(z)$.

Hence, we  can use Lemma \ref{lema} to find a constant $K_6$, such that for every $z\in \tilde{\Lambda}(p,N)$ and $v\in \tilde{E}^u(z)$,
\begin{equation}
|Dg^r(z)(v)|\geq (C\, K_6)^l \, \|Dg_p^{-k}\|^{-l}|v|, \text{ for } r=l(k+T),\; l\in \N. \label{equa 2}\end{equation} where
$$C=\inf_{z\in V\backslash g^{-1}(V), \; |v|=1} \|Dg^{T}(z)(v)\|.$$

Therefore, it's not difficult to see that for $N$ big enough, all points in $\tilde{\Lambda}(p,N)$ have positive lyapunov exponentes bigger than
$\chi(p,g)-1/n$. In particular, we can choose $N_3$ in order to get $k>>T$, such that for any periodic point $q\in \Lambda(p,N)$, $N>N_3$,

$$
\chi(q,g)\geq \chi(p,g)-\frac{1}{n}.
$$

Hence, if we take  $\Lambda(p,n)=\Lambda(p,N)$ for $N=\max\{N_1,\, N_2,\, N_3\}$, the properties of proposition are satisfied for the
perturbation $g$ of $f$.

Now, by robustness of the hyperbolic periodic point $p$ and the set $\Lambda(p,n)$, properties  (1) and (2) are also satisfied for
diffeomorphisms close to $g$. Recall that $\mu_n$ is the one that maximize topological entropy, then in order to prove properties (3) and (4) we
just concerned  with some neighborhood of the set $\Lambda(p,n)$, and so, the same could be done for diffeomorphisms near $g$. Which concludes
the proof of proposition.

$\hfill\square$

\section{Example of discontinuity points for topological entropy}

In order to prove Theorem C, we construct an example of a $C^{\infty}$ area preserving diffeomorphism over $S^2$ that is a non upper
semi-continuity point for topological entropy in the space $\diff_{\omega}^1 (S^2)$.

Firstly, let $S$ be any surface different of $\mathbb{T}^2$. As $S$ does not accept Anosov diffeomorphism,  by Theorem A we have that for
generic volume preserving diffeomorphisms of $S$
$$h(f)=s(f).$$ Hence, using that $s(.)$ is a lower semi-continuous function, if we find a diffeomorphim $f\in \diff^1_{\omega}(S^2)$ such that
$h(f)<s(f)$, then this is an example where topological entropy is not upper semi-continuous.

In order to find such diffeomorphism, we use the following result of Lai-Sang Young \cite{young}.

\begin{theorem}
Let $\phi: \R\times M\rightarrow M$ be a flow in a  2-dimensional manifold $M$. Then, the diffeomorphism $\phi_t=\phi(t,.)$ over $M$ has zero
topological entropy, i.e, $h(\phi_t)=0$, for every t.
\end{theorem}

By this result and the above discussion, to prove Theorem C it's enough to find a hamiltonian flow in $S^2$ with a hyperbolic periodic
orbit. 

In what follows we describe the construction of such example which uses the well known mathematical pendulum, see \cite{palis takens}. Recall
that a vector field $X_H$ over a compact symplectic manifold $(M,\omega)$ is Hamiltonian iff there exists a smooth map $H: M\rightarrow \R$ such
that
$$
\omega(X_H,.)=dH.
$$
Also, recall  $\phi_t=\phi(t,.):M\rightarrow M$ is a symplectic diffeomorphism, for every $t\in \R$, where $\phi$ is the flow generated by
$X_H$. Note that in dimension two the space of conservative diffeomorphisms coincides with the symplectic one.

From now on we consider the symplectic manifold $(S^2, \omega)$ the two dimensional sphere, with $\omega(x)=<x,u\times v>$, for $x\in S^2$ and
$u,v\in T_x S^2$, some induced area form over $S^2$. If we give for $S^2$ cylindrical polar coordinates $(\theta, z)$, $0\leq\theta< 2\pi$ and
$-1<z<1$, away from its poles, we can verify that $\omega=d\theta \wedge d z$.

Let $H_1(\theta, z)=z$ be the height function over the sphere, and $X_{H_1}$ be the Hamiltonian vector field generated by $H_1$. Note the flow
generated by $X_{H_1}$ has no hyperbolic periodic orbits, more precisely the poles are non-hyperbolic singularities, and the flow far from them
is $\phi(t,(\theta,z))=(\theta+t,z)$, i.e., rotations.

On the other hand, we can use the famous mathematical pendulum on $S^1\times \R$ to build hyperbolic periodic orbits in the previous flow.  Let
$H_2: S^1\times \R\rightarrow \R$ be the total energy of the pendulum, $H_2(\theta,z)=\frac{1}{2} z^2-\cos \theta$, then the Hamiltonian vector
field $X_{H_2}$ on the cylinder gives us the phase portrait of the pendulum. We observe that the flow generated by $X_{H_2}$ has an unstable
equilibrium at $p=(\pi, 0)$. Now, considering $\beta: (-1,1)\rightarrow \R$ the $C^{\infty}$ bump function such that $\beta(x)=1$ if $|x|<1/2$
and $\beta(x)=0$ if $|x|>2/3$, we define $H: S^1\times (-1,1)\rightarrow \R$ as follows
$$
H(\theta, z)=\beta(|z|)H_2(\theta,z)+ (1-\beta(|z|)) H_1(\theta,z).
$$

Hence, after some coordinate change we can look for this function over $S^2$. In fact, what we did was just to carry the pendulum flow to the
sphere by changing the height function on some strip. See figure \ref{exemplo}. And finally,  $X_H$ is a Hamiltonian vector field on $S^2$ and
the flow generated by it has a hyperbolic singularity as we wanted.

\begin{figure}[httb]
\vsm\centering
\includegraphics[scale=0.7]{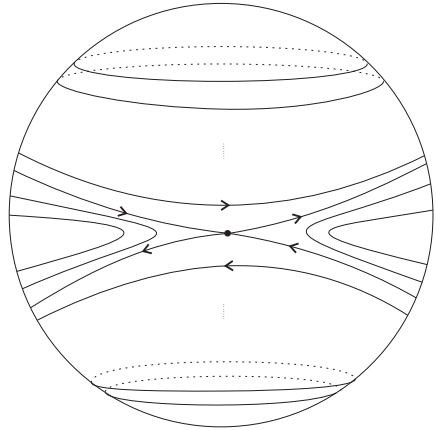}
\caption{Phase portrait of $X_H$}\vsm\label{exemplo}\end{figure}

{\bf Acknowledgements:}
T. C was
supported by CAPES-Brazil and FAPESP-Brazil by a Doctoral fellowship. A. T was supported by FAPESP-Brazil and CNPq-Brazil. The authors also wish to point out the excellent research atmosphere at ICMC-USP.

\medskip

\vspace{0,3cm} \flushleft
Thiago Catalan\\
Instituto de Ci\^encias, Matem\'atica e Computa\c{c}\~ao \\
Universidade de S\~ao Paulo\\
16-33739153 S\~ao Carlos-SP, Brazil\\
E-mail: catalan@icmc.usp.br
\flushleft
Ali Tahzibi\\
Instituto de Ci\^encias, Matem\'atica e Computa\c{c}\~ao \\
Universidade de S\~ao Paulo\\
 S\~ao Carlos-SP, Brazil\\
 Email: tahzibi@icmc.usp.br, 
 URL: www.icmc.usp.br/$\sim$tahzibi


\begin{thebibliography}{00}






\addcontentsline{toc}{chapter}{Referências Bibliográficas}


\bibitem{abc} F. Abdenur, C. Bonatti and S. Crovisier,  Nonunifom Hyperbolicity for $C^1-$generic diffeomorphisms, {\em preprint} (2008) {\tt arXiv:0809.3309},
 To appear in Israel Journal of Mathematics.

\bibitem{Arbieto-catalan} A. Arbieto and T. Catalan, Hyperbolicity in the Volume Preserving Scenario {\em preprint} (2010) {\tt arXiv:1004.1664}.

\bibitem{pasting lema} A. Arbieto and C. Matheus, A pasting lemma and some applications for conservative systems, {\it Erg. Th. and Dynamic. Sys}, 27 (2007), 1399-1417.

\bibitem{Arnaud} M-C. Arnaud,  Le "Closing Lemma" en topologie $C^1$, {\it Supplément au Bull. Soc. Math. Fr.}, 74(1998)

\bibitem{avila2008} A. Avila, On the regularization of conservative maps, {\em preprint} (2008) {\tt arXiv:0810.1533}, To appear in Acta Matematica.


\bibitem{Bowen} R. Bowen, Topological Entropy and Axiom A, {\it Proc. Symp. Pure Math., AMS., Providence RI.,} 14 (1970), 23-41.

\bibitem{Boyle} M Boyle, D. Fiebig and U. Fiebig, Residual entropy, conditional entropy, and subshift covers, {\it Forum Math.}, 14 (2002), 713-757.

\bibitem{Burguet}D. Burguet,  $C^2$ surface diffeomorphisms have symbolic
extensions, {\em preprint} (2010) {\tt arXiv:0912.2018}.


\bibitem{Buzzi} J. Buzzi, Intrinsic ergodicity for smooth interval maps, {\it Israel J. Math.}, 100 (1997), 125-161.


\bibitem{df} L. Diaz and T. Fisher, Symbolic extensions for partially hyperbolic diffeomorphisms, {\em preprint} (2009) {\tt arXiv:0906.2176},
To appear in Discrete and Continuous Dynamical Systems.


\bibitem{diazfisher} L. Diaz, T. Fisher, M. Pac\'{i}fico, and J. Vieitez, Entropy-expansiveness for partially hyperbolic diffeomorphisms, {\em preprint} (2010)
{\tt arXiv:1010.0721}.



\bibitem{horitatahzibi} V. Horita and A. Tahzibi, Partial hyperbolicity for symplectic diffeomorphisms, {\it Ann. I. H. Poincaré – AN} 23
(2006), 641–--661.

\bibitem{katok} A. Katok, Lyapunov exponents, entropy and periodic points for diffeomorphisms, {\it Publications Mathématiques de l´IHES}, 1980.


\bibitem{LLS} C. Liang, G. Liu and W. Sun, Equivalent Conditions of Dominated Splittings for Volume-Preserving Diffeomorphism, {\it Acta Math. Sinica}
 23 (2007), 1563-1576.

\bibitem{GYV} G. Liao, J. Yang and M. Viana, Entropy of diffeomorphisms away from tangencies , {\it Private communication}.
\bibitem{Newhouse3} T. Downarowicz and S. E. Newhouse, Symbolic extension and smooth dynamical systems, {\it Inventiones Mathematicae}, 160 (2005), 453-499.




\bibitem{Newhouse1} S. E. Newhouse, Topological entropy and Hausdorff dimension for area preserving diffeomorphisms of surfaces,
{\it Société Mathématique de France, Astérisque}, 51  (1978), 323-334.


\bibitem{Newhouse2}  S. E. Newhouse,  Quasi-eliptic periodic points in conservative dynamical systems,
{\it American Journal of Mathematics}, 99, No. 5  (1975), 1061-1087.

\bibitem{palis takens} J. Palis  and F. Takens, Hyperbolicity and sensitive-chaotic dynamics at homoclinic bifurcations. {\it Cambridge: Cambridge University Press},
1993. (Cambridge Studies in Advanced Mathematics).



\bibitem{Rees}  M. Rees, A minimal positive entropy homeomorphism of the $2$-torus J. London Math. Soc. (2)  23  (1981), no. 3, 537--550
\bibitem{Tahzibi} A. Tahzibi,  $C^1$ Generic Pesin's entrpy formula, {\it C. R. Acad. Sci. Paris, Ser. I} 335 (2002), 1057-1062.



\bibitem{Xia} Z. Xia, Homoclinic points in symplectic and Volume-Preserving diffeomorphisms, {\it Communications in Mathematical Physics},
177 (1996), 435-449.

\bibitem{young} L. S. Young,  Entropy of continuous flows on compact 2-manifolds, {\it Topology}, 16  (1977), no. 4, 469–471.

\bibitem{zehnder} E. Zehnder,  Note on smoothing symplectic and volume-preserving diffeomorphisms, {\it Geometry and topology}
(Proc. III Latin Amer. School of Math., Inst. Mat. Pura Aplicada CNPq, Rio de Janeiro, 1976), pp. 828–854. Lecture Notes
in Math., Vol. 597, Springer, Berlin, 1977.

\end{thebibliography}
\end{document}